\documentclass[12pt]{amsart}

\usepackage{amsmath}
\usepackage{amssymb}
\usepackage{latexsym}
\usepackage{verbatim}
\usepackage{graphics}

\def\NN{\mathbb{N}}
\def\RR{\mathbb{R}}
\def\ZZ{\mathbb{Z}}
\def\card{{\rm Card}}

\def\MSE{{\rm MSE}}
\def\MSEprime{{\rm MSE}^{'}}
\def\WSE{{\rm WSE}}
\def\St{{\rm St}}

\def\M{\mathcal M}

\newtheorem{theo}{Theorem}
\newtheorem{defi}[theo]{Definition}
\newtheorem{prop}[theo]{Proposition}
\newtheorem{lemma}[theo]{Lemma}
\newtheorem{corollaire}[theo]{Corollary}

\long\def\STOP#1\GO{{\bf\par Ici passage sautŽ\par}}
\begin{document}

\title{Words and morphisms with Sturmian erasures}

\author{Fabien Durand}

\author{Adel Guerziz}

\author{Michel Koskas}

\address{Laboratoire Ami\'enois
de Math\'ematiques Fondamentales  et
Appliqu\'ees, CNRS-UMR 6140, Universit\'{e} de Picardie
Jules Verne, 33 rue Saint Leu, 80039 Amiens Cedex 1, France.}

\email{fabien.durand@u-picardie.fr}

\email{adel.guerziz@u-picardie.fr}

\email{koskas@laria.u-picardie.fr}


\begin{abstract}

We say $x \in \{ 0,1,2 \}^{\NN}$ is a word with Sturmian erasures if for
 any $a\in \{ 0,1,2 \}$ the word obtained erasing all $a$ in $x$ is a Sturmian word.
A large family of such words is given coding trajectories of balls in
 the game of billiards in the cube. We prove that the monoid of
 morphisms mapping all words with Sturmian erasures to words with
 Sturmian erasures is not finitely generated.
\end{abstract}

\maketitle

\section{Introduction}

In this paper we are interested in words $x$ defined on the alphabet
$A_3 = \{ 0,1,2 \}$ having the following property: For any letter $a\in A_3$, the word obtained erasing all $a$ in $x$ is a Sturmian word. We say $x$ is a {\it word with Sturmian erasures}.

Sturmian words are well-known objects that can be defined in many
ways. For example, a word is Sturmian if and only if for all $n\in
\NN$ the number of distinct finite
words of length $n$ appearing in $x$ is $n+1$ (see \cite{Lo} for complete
references about Sturmian words). Sturmian words can also be viewed as
trajectories of balls in the game of billiards in the square. We will see that a large
family of words with Sturmian erasures is the family of trajectories of
balls in the game of billiards in the cube.

Here we are interested in the morphisms $f:A_3\to A_3^*$ (the free monoid generated by $A_3$) that send all
words with Sturmian erasures to words with Sturmian erasures. We call
such $f$ the {\it morphisms with Sturmian erasures} and we denote by $\MSE$
the set of all these morphisms. Our main result is the following:

\begin{theo}
\label{maintheo}
We have:
\begin{enumerate}
\item
\label{notfinitely}
The monoid $\MSE$ is not finitely generated;
\item
\label{union}
$\MSE$ is the union of $\MSE^\varepsilon$ and the set of permutation of
     $A_3$.
\item
\label{locally}
If $f : A_3 \to A_3^*$ is locally with Sturmian erasures such that $f(i)$ is the empty word for
     some $i\in A_3$ then it is a morphism with Sturmian erasures;
\end{enumerate}
\end{theo}
Where $\MSE^\varepsilon$ is the set of morphisms with Sturmian erasures having the
empty word as an image of a letter and locally with Sturmian erasures
means that there exists a word with Sturmian erasures such that $f(x)$
is a word with Sturmian erasures.

We recall that F. Mignosi and P. S\'e\'ebold proved in \cite{MS} that the monoid of the morphisms sending all Sturmian words to Sturmian words is finitely generated.

In the last section we give some other informations about the words
with Sturmian erasures: symbolic complexity, link with the game of billiards in
the cube, balanced property and palyndroms.

\section{Definitions, notations and background}

\subsection{Words, morphisms and matrices}

We call {\it alphabet} a finite set of elements called {\it
letters}. Let $A$ be an alphabet and $A^*$ be the free monoid generated
by $A$. The elements of $A^*$ are called {\it words}. The neutral
element of $A^*$, also called the {\it empty word}, is denoted by
$\varepsilon$. We set $A^+ = A^* \setminus \{\varepsilon\}$. Let $u =
u_0u_1 \cdots u_{n-1}$ be a word of $A^*$, $u_i\in A$, $0\leq i\leq
n-1$. Its {\it length} is $n$ and is denoted by $|u|$. In
particular, $|\varepsilon| = 0$. If $a \in A$ then $|u|_a$ denotes the
number of occurrences of the letter $a$ in the word $u$. We call {\it
infinite words} the elements of $A^{\NN}$ and we set $A^{\infty}
= A^{\NN} \cup A^*$. Let $x \in A^{\infty}$ and $y \in  A^*$. We say
that $y$ is a {\it factor} of $x$ if there exist $u \in  A^{*}$ and  $v \in  A^{\infty}$ such that
$x=uyv$. In particular if $u=\varepsilon$ then $y$ is a {\it prefix} of
$x$ and if $v=\varepsilon$ then $y$ is a {\it suffix} of $x$.
An infinite word $x=(x_n ; n\in \NN)$ of $A^\NN$ is called {\it eventually periodic}
if there exist two words $u \in A^*$ and $v \in A^+$ such that $x =
uvvv\dots$.

The {\it complexity} function of an infinite word $x$ is the function
$P_x : \NN \rightarrow \NN$ where $P_x (n)$ is the number of factors of length $n$ of $x$.

Let $A$, $B$ and $C$ be three alphabets. A {\it morphism} $f$ is a
map from $A$ to $B^*$. It induces by concatenation a map from
$A^*$ to $B^*$. If $f (A)$ is included in $B^+$, it induces a map
from $A^{\NN}$ to $B^{\NN}$. All these maps are also written $f$.

%We
%say a set $\M $ of morphisms from $A$ to $A^{*}$ is a {\it monoid of
%morphisms} if it is a monoid for the composition law of morphisms.

To a morphism $f : A\to B^*$ is associated the matrix
$M_{f} = (m_{i,j})_{i\in  B  , j \in  A  }$ where $m_{i,j}$ is the
number of occurrences of $i$ in the word $f(j)$. If $g$ is a morphism
from $B$ to $C^*$ then we can check we have $M_{g\circ f} = M_g M_f$.

\subsection{Sturmian words and Sturmian morphisms}

Let $A$ be a finite alphabet. An infinite word $x\in A^{\NN}$ is {\it Sturmian} if for all $n \in
\mathbb{N}, P_x(n) = n+1$. Since $P_x (1) = 2$, we can suppose
$A=\{0, 1\} $ (see \cite{Lo} for more informations about these words).

A morphism $f$ from $A$ to $A^*$ is {\it Sturmian} if for all Sturmian
word $x$ the word $f(x)$ is Sturmian. A morphism $f$ is {\it locally Sturmian} if
there exists at least a Sturmian word $x$ such that $f(x)$ is
Sturmian. We call $\St$ the semigroup generated by the morphisms $E$,
$\varphi$, and $\widetilde{\varphi}$ defined by
\begin{center}
\begin{tabular}{llllllllllllll}
$E:$ & $A^*$ & $\longrightarrow $ & $A^*$ & & $\varphi :$ & $A^*$ &
 $\longrightarrow $ & $A^*$ & & $\widetilde{\varphi}:$ & $A^*$ &
 $\longrightarrow $ & $A^*$ \\
     & $0$   & $\longmapsto $     & $1$   & &             & $0$   &
 $\longmapsto $     & $01$  & &                        & $0$   &
 $\longmapsto $   & $10$ \\
     & $1$   & $\longmapsto $     & $0$   & &             & $1$   &
 $\longmapsto $     & $0$   & &                        & $1$   &
 $\longmapsto $   & $0$
\end{tabular}
\end{center}

\begin{theo}
\cite{BS,MS}
\label{BS93}
The following three conditions are equivalent
\begin{enumerate}
\item $ f \in \St $;
\item $ f $ locally Sturmian;
\item $ f $ Sturmian.
\end{enumerate}
\end{theo}

\subsection{Words with Sturmian erasures}

\label{defg}

Let $A_3 = \{0, 1,2 \}$ and
let $x$ be a infinite word of $A_3^{\NN}$. For $i\in A_3$ we denote $\pi_i : A_3 \to A_3^* $
the morphism defined by $\pi_i (j) = j$ if $j\in A_3 $ with $j\not = i$
and $\pi_i (i) = \varepsilon$.

\begin{defi}
An infinite word $x\in A_3^{\NN}$ is called {\rm word with Sturmian erasures} if and
 only if the word
 $\pi_i(x)$ is a Sturmian word for all $i\in A_3$.
We say \linebreak $f : A_3 \longrightarrow A_3^*$ is a {\rm morphism with Sturmian
 erasures} if $f (x)$ is a word with Sturmian erasures for all words
 $x\in A_3^{\NN}$ with
 Sturmian erasures.
\end{defi}

We call $\WSE$ the set of words with Sturmian erasures and $\MSE $ the
set of morphisms with Sturmian erasures. We remark $\MSE$ is a monoid for the composition law of
morphism. The image of a Sturmian word by a morphism with Sturmian
erasures is a word with Sturmian erasures. Hence $\WSE$ is not empty.

\medskip

{\bf Example 1.} Let $g : A_3 \rightarrow A_3^*$ be the morphism defined
by : $g(0)=02$, $g(1)=10$ and $g (2) = \varepsilon$. Let $F_0 = 0$ and for $n\geq 0$ $F_{n+1} = \varphi (F_n)$. Let $F\in \{ 0,1 \}^{\NN}$ be the unique fixed point of $\varphi$ in
$\{ 0,1 \}^{\NN}$ (see \cite{Qu}). Then for each $n\geq 0$ $F_n$ is a prefix of $F$, and we have $F= 0100101001001...$. This word is called the {\it Fibonacci word} (remark that $|F_{n+2}| = |F_{n+1}| + |F_n|$, $n\geq 0$). It is a Sturmian
word. From Theorem \ref{BS93} we deduce that
$$
g (F) = 0210020210021002021002021002100202100210 \dots
$$
is a word with Sturmian erasures. Hence $\WSE$ is not empty.

Let $x\in
A_3^{\NN}$ be a word with Sturmian erasures. The word $\pi_2 (x)$ is a
Sturmian word and $g\circ \pi_2 = g$. Moreover $\pi_2 \circ g_{| \{ 0,1\} }$ is a Sturmian
morphism. Hence
$$
\pi_2 \circ g (x)
=
\pi_2 \circ g \circ \pi_2 (x)
=
\pi_2 \circ g_{ | \{ 0,1 \}} ( \pi_2 (x))
=
g_{ | \{ 0,1 \}} ( \pi_2 (x))
$$
is a Sturmian word. We can also show that $\pi_0 \circ g (x) $ and
$\pi_1 \circ g (x) $ are words with Sturmian erasures. Hence, $g$ is a
morphism with Sturmian erasures and $\MSE$ is not empty.

\section{Proofs of points (\ref{union}) et (\ref{locally}) of Theorem
\ref{maintheo} }

We denote by $\MSE^\varepsilon$ the set of morphisms of $\MSE$ such that there exists $l \in A_3$ with
 $f(l)=\varepsilon$. We will prove that $\MSE$ is the union of
 $\MSE^\varepsilon$ with the set of permutations on $A_3$. This last set is
 generated by

\begin{center}
\begin{tabular}{llllllllllllllllllll}
$E_0:$ & $A_3^{*}$ & $\longrightarrow $ & $A_3^{*}$ & $E_1:$ & $A_3^{*}$
 & $\longrightarrow $ & $A_3^{*}$ & $E_2:$ & $A_3^{*}$ &
 $\longrightarrow $ & $A_3^{*}$\\

       & $0$       & $\longmapsto$&$0$       &        & $0$       &
 $\longmapsto $       & $2$       &        & $0$       & $\longmapsto $ & $1$ \\

       & $1$       & $\longmapsto$&$2$       &        & $1$       &
 $\longmapsto $     & $1$         &        & $1$       & $\longmapsto $ & $0$\\

       & $2$       & $\longmapsto$&$1$       &        & $2$       &
 $\longmapsto $     & $0$         &        & $2$       & $\longmapsto $ & $2$
\end {tabular}
\end{center}

We denote by
 $\MSE_i$, $i \in A_3$, the set of morphisms with Sturmian erasures such
 that $f(i)=\varepsilon$. We have $\MSE^\varepsilon = \MSE_0 \cup
 \MSE_1 \cup \MSE_2$.

\subsection{Proof of the point (\ref{locally}) of Theorem \ref{maintheo}}

We start with the following proposition.

\begin{prop}
\label{local}
If a morphism $f: \{ 0,1 \} \rightarrow A_3^*$ maps a Sturmian word
 defined on $\{ 0,1 \}^{\NN}$ into a
 word with Sturmian erasures then it maps any Sturmian word into a word with Sturmian erasures.
\end{prop}
\begin{proof}
Let $x\in \{ 0,1 \}^{\NN}$ be a Sturmian word and $f: \{ 0,1 \} \rightarrow A_3^*$ be a
 morphism such that $f(x)$ is a word with Sturmian erasures.

Let $i$ be a letter of $A_3$ and $\overline{E} : A_3 \rightarrow \{ 0,1 \}^*$
 be a morphism such that \linebreak $\{ 0,1 \} = \{ \overline{E} (a) ; i\not = a, \ a\in A_3\}$.
Then $\overline{E} \circ \pi_i \circ f (x)$ is Sturmian and \linebreak $\overline{E} \circ \pi_i \circ f : \{ 0,1 \} \rightarrow \{ 0,1 \}^*$ is a locally Sturmian morphism. Hence, from Theorem \ref{BS93} it is Sturmian. It follows that for every Sturmian word $y$, $f(y)$
is a word with Sturmian erasures. This ends the proof.
\end{proof}

Now we prove the point (\ref{locally}) of Theorem \ref{maintheo}.
Let $f : A_3 \to A_3^*$, such that $f(i) = \varepsilon$ for some $i\in
A_3$, and $x\in A_3^{\NN}$ be a word with
 Sturmian erasures such that $f (x)$ is a word with Sturmian erasures.
We remark that we have $f\circ \pi_i (x) = f (x)$ and that $\pi_i (x)$
 is a Sturmian word.

We can suppose $A_3\setminus \{i\} = \{ 0,1 \}$. Hence the morphism $f\circ {\pi_i}_{|\{ 0,1 \}}$ satisfy the
 hypothesis of Proposition \ref{local}. Consequently if $y$ is a word
 with Sturmian erasures then $f (y) = f\circ \pi_i (y) = f\circ {\pi_i}_{|\{ 0,1 \}} (\pi_i (y)) $ is a word with Sturmian erasures.
\hfill $\Box$

\medskip

{\bf Example 2.} We can remark that there exist morphisms $f : A_3 \rightarrow A_3^*$
 such that for some word $x\in \WSE $ we have $f(x)\in \WSE$ but
 $f $ is not a morphism with Sturmian erasures.

For example let $F$ be the Fibonacci word, $f$ be defined by $f(0) = 0$, $f(1) = 1$ and $f(2) =
 012$, $g : A_3 \rightarrow A_3^*$ be defined
 by $g(0) = 01$, $g(1) = 02$ and $g (2) = \varepsilon$, and, $h : A_3 \rightarrow A_3^*$ be defined
 by $h(0) = 02$, $h(1) = 10$ and $h (2) = \varepsilon$.

As in Example 1 we can prove that $g$, $h$ and $f\circ g$ are morphisms
with Sturmian erasures and consequently $x = g (F)$ and $ f (x) = f\circ
g (F)$ are words with Sturmian erasures.

But we remark that $f\circ h(F)$ is not a word with Sturmian erasures. Indeed
 001210 is a prefix of $f(y)$ and 00110 is a prefix of $w = \pi_2 f (y)$. Consequently $P_w (2)=4$ and $w$ is not a Sturmian word.

\subsection{Proof of the point (\ref{union}) of Theorem \ref{maintheo}}
\label{almost}

We need the
following lemma that follows from Theorem \ref{BS93} and the fact that the determinant of the
 matrices associated to $\varphi$, $\widetilde{\varphi}$ and $E$ belong to $\{-1, 1\}$.

\begin{lemma}
\label{ch5.160}
Let $M$ be the matrix associated to the Sturmian morphism $f$. Then
 $\det M =\pm 1$.
\end{lemma}

Let us prove the point (\ref{union}) of Theorem \ref{maintheo}. This
proof is due to D. Bernardi.

Let $f$ be a morphism of $ \MSE$.
Let $i\in A_3$ and $\overline{E} : A_3 \setminus \{ i \} \rightarrow
\{ 0,1 \}^*$ be a morphism such that $\{ 0,1 \} = \{ \overline{E} (a) ; a\not = i , \ a\in A_3 \}$. We set $h=\overline{E} \circ \pi_i \circ f\circ g_{|\{ 0,1 \}}$ where $g : A_3 \rightarrow A_3^*$ is the
 morphism defined by: $g(0)=02$, $g(1)=12$ and $g (2) = \varepsilon$.

As the letter $i$ does not appear in the images of $\pi_i \circ f$, we
 consider $\pi_i \circ f$ as a morphism from $A_3$ to $(A_3 \setminus \{
 i \})^*$. We set $M_{\pi_i \circ f} = (u, v, w)$ where $u$, $v$ and $w$
 are column vectors belonging to $\mathbb{R}^2$. We recall $g$ is a morphism with Sturmian erasures (see Example 1 of the subsection \ref{defg}). Hence the morphism $h$ is Sturmian and we have $M_h=(u+w, v+w)$. From Lemma \ref{ch5.160}

$$
\det (u, v) + \det(u, w) + \det (w, v)
=
\det (u+w, v+w)
=
\pm 1 .
$$

We do the same with $g$ being one of the two following morphisms :
($0\longmapsto01$, $1\longmapsto12$, $2\mapsto \varepsilon$) and ($0\longmapsto02$,
 $1\longmapsto01$, $2\mapsto \varepsilon$). We obtain finally

$$
(1) \ \ \det (u, v)+\det (u, w) + \det (v, w) = \pm1.
$$
$$
(2) \ \ \det (u, v) + \det (w, u) + \det (w, v) = \pm1.\\
$$
$$
(3) \ \ \det (u, v) + \det (u, w) + \det (w, v) = \pm1.
$$

The combinations of the equations (1) and (2), (2) and (3), and, (3) and
 (1) imply respectively that $\det (u, v)$, $\det (w, u)$ and $\det (w,
 v)$ belong to $\{ -1,0,1 \}$.
From (1), one of three determinants $\det (u, v)$, $\det (u, w)$ or
 $\det (v, w)$ is different from $0$.

We suppose $\det (u, v) \neq 0$ (the other cases can be treated in the
 same way). The set $\{ u, v\}$ is a base of $\mathbb{R}^2$, hence there
 exist two real numbers $a$ and $b$ such that $w=au+bv$. We have $a = \det (w,v)/\det (u,v)$ and $b = \det (w,u) /\det (v,u) $.
Moreover from (1) and (2) we see that $\det (u,w) + \det (v,w)$ and
 $-(\det (u,w) + \det (v,w)) = \det (w,u) + \det (w,v)$ belong to $\{ \det (v,u) - 1 , \det (v,u) + 1\}$ which is equal to $\{ 0,2\}$ or $\{ -2 , 0 \}$. Consequently
$\det (u,w) + \det (v,w) = 0$. Hence $a=b$ and $w = a(u+v)$. The vector
 $w$ is the column of the matrix of a morphism therefore it has
 non-negative coordinates which implies that $a$ is non-negative.

Suppose $a>0$. Then $\det (w,v)$ and $\det (u,v)$ are positive and have the same
 sign. Hence one of the equations (2) or (3) is equal to -3 or 3 which
 is not possible. Consequently $a=0$ and $w=0$.

Therefore for all $i\in A_3$ the matrix $M_{\pi_i \circ f} = (m_i (c,d))_{c\in A_3 \setminus \{ i \} ,  d\in A_3}$ has a column  $(m_i (c,d_i))_{c\in A_3\setminus \{ i \}}$
 with entries equal to 0.

Two cases occurs.

1- There exists $i, j \in A_3$ ($i\neq j$) such that $d_i=d_j$. In this
 case we easily check that $f (d_i) = \varepsilon$. Consequently $f$
 belongs to $\MSE^\varepsilon $.

2- The sets $\{d_0, d_1, d_2\}$ and $A_3$ are equal. In this case we can
 check that $f$ is a permutation.
\hfill $\Box$

\section{Prime morphisms}

\subsection{Some technical definitions}

Let $A$ be an alphabet and $f : A\rightarrow A^{*}$ be a
morphism. A letter $a$ is called {\it $f$-nilpotent}  if there exists an
integer $n$ such that $f^n(a) = \varepsilon$ (if it is not ambiguous we
will say it is nilpotent). The set of $f$-nilpotent letters is denoted
by ${\mathcal N}_f$. We call ${\mathcal P}_f^{'}$ the set of letters $a$
such that there exists an integer $n$ satisfying $\pi_{{\mathcal N}_f}(f^n(a)) = a$
   where $\pi_{{\mathcal N}_f} (b) = \varepsilon $ if ${b\in \mathcal N}_f$ and $b$ otherwise. The set of such letters is denoted by ${\mathcal P}_f^{'}$.

We say the letter $a$ is {\it $f$-permuting} if there exists an integer $n$ such that \linebreak $f^n(a) \in ({\mathcal N}_f \cup {\mathcal P}_f^{'})^* \setminus {{\mathcal N}_f}^*$. We denote by ${\mathcal P}_f$ the set of such letters.
We remark that ${\mathcal P}_f^{'}$ is included in ${\mathcal P}_f$.

A letter $a$ is called {\it $f$-expansive}, or expansive when the
context is clear, if it is neither nilpotent
nor permuting.
We remark the letter $a\in A$ is $f$-expansive if and only if
$\lim_{n\rightarrow +\infty} |f^n (a)| = +\infty $ and it is $f$-permuting if and only if the sequence $( |f^n (a)| ; n\in \NN )$ is bounded and is never equal to 0.

The morphism $f$ is {\it nilpotent} if $f (A)$ is included in ${\mathcal
N}_f^{*}$, i.e., if there exists an integer $n$ such that $f^n (a) =
\varepsilon $ for all $a\in A$. A morphism $f$ is called {\it expansive} if
there exists a $f$-expansive letter. A morphism $f$ is a {\it unit} if
it is neither nilpotent nor expansive. In others words if
$f (A)$ is included in $({\mathcal N}_f \cup {\mathcal P}_f)^{*}$.

Let $\M$ be a monoid of morphisms. A morphism $f\in \M$ is said to be
 {\it prime} in $\M$ if for any
 morphisms $g$ and $h$ in $\M$ such that $f=g \circ h$, then $g$ or $h$
 is a unit of $\M$.
We say that $f$ is of {\it degree} $n$ in $\M$, $n \in \mathbb{N}$, if any
 decomposition of $f$ into a product of prime and unit morphisms of $\M$
 contains at least $n$ prime morphisms and there exists at least one decomposition of $f$ into
 prime and unit morphisms of $\M$ containing exactly $n$ prime morphisms.

The set of prime morphisms in ${\rm St}$ is $\{ f\circ g \circ h ; f,h\in \{
Id , E \}, \ g\in \{ \varphi , \widetilde{\varphi } \} \}$.

\subsection{Some conditions to be a prime morphism}
In the sequel we need the following morphisms which are
extensions to the alphabet $A_3$ of the morphisms $\varphi$ and $\widetilde{\varphi}$:

\begin{center}
\begin{tabular}{llllllll}
$\varphi_1:$ & $A_3^{*}$ & $\longrightarrow $ & $A_3^{*}$ \ , & \hspace{1cm} $\widetilde{\varphi_1}:$ & $A_3^{*}$ & $\longrightarrow$ & $A_3^{*}$\\

             & $0$       & $\longmapsto $     & $01$      &
      & $0$       & $\longmapsto $    & $10$ \\

             & $1$       & $\longmapsto $     & $0$       &
      & $1$       & $\longmapsto $    & $0$  \\

             & $2$       & $\longmapsto $ & $\varepsilon $ &
      & $2$       & $\longmapsto $& $\varepsilon$      .
\end {tabular}
\end{center}

In the next proposition we need the following lemma.

\begin{lemma}
\label{longueur}
Let $g\in \MSE_2$. Then,
$|g(a)| \geq 2$, $|g(a)|_0 + |g(a)|_1 \geq 1$ and \linebreak $|g (01)|_a \geq 1$
for all $a\in \{ 0 ,1 \}$.
\end{lemma}
\begin{proof}
Suppose for $a\in \{ 0,1 \}$ we have $|g(a)| = 1$, for example $g (a) = b$. Then $\pi_b \circ g (x)$ is periodic for all $x\in A_3^{\NN}$. This contradicts the fact that $g$ belongs to $\MSE_2$. If $|g(a)| = 0$ we have the same conclusion. This proves the first part of the lemma.

Suppose $|g(a)|_0 + |g(a)|_1 = 0$. Then $\pi_2 \circ g (x)$ is periodic
 for all $x\in A_3^{\NN}$. This proves the second inequality.

Suppose $|g (01)|_a = 0$, then $a$ does not appear in $g (x)$. This contradicts the fact that $g$ belongs to $\MSE_2$.
\end{proof}

\begin{prop}
\label{condpremier}
Let $i\in A_3$ and $ f \in \MSE_i$. We set $A_3= \{ i,j,k \}$.

1) If $f(j)$ is
 neither a prefix nor a suffix of $f(k)$ and that $f(k)$ is neither a prefix
 nor a suffix of $f(j)$, then $f$ is prime in
 $\MSE_{i}$.

2) Moreover, if $f$ is prime in $\MSE_{i}$ and if we have
$
|f(012)|_j > |f(012)|_k \geq |f(012)|_i,
$
then $f(j)$ is
 neither a prefix nor a suffix of $f(k)$ and $f(k)$ is neither a prefix
 nor a suffix of $f(j)$.
\end{prop}

\begin{proof}
We only make the proof in the case $i=2$.

1) We suppose $f(0)$ is
 neither prefix nor suffix of $f(1)$ and that $f(1)$ is neither prefix
 nor suffix of $f(0)$.
We proceed by contradiction, i.e., we suppose there exist $ g,h \in \MSE_2$ which are not units such that $f=g \circ h$.

Let $h_1=\pi_2 \circ h$. We have $g \circ h=g \circ h_1$. We define
 $\overline{h}_1 : \{ 0,1 \} \rightarrow \{ 0,1 \}^*$ by $\overline{h}_1 (i) =h_1 (i)$ for
 all $i\in \{ 0,1 \}$. We remark that $\overline{h}_1$ is a Sturmian morphism hence it is a
 product of $\varphi, \widetilde{\varphi}$ and $E$ (Theorem \ref{BS93}). Therefore $h_1$ is a
 product of
$\varphi_1, \widetilde{\varphi}_1$ and $\pi_2 \circ E_2$.

We consider two cases.

Suppose $h_1 \not \in  \{ \pi_2\circ E_2 , \pi_2\circ E_2 \circ E_2 \}$. Then, for example, $h_1$ is equal to
 $h_2 \circ \varphi_1$ where $h_2$ is a product of $\varphi_1, \widetilde{\varphi}_1$ and $\pi_2 \circ E_2$. The other cases ($h_1=h_2 \circ \widetilde{\varphi_1}$ or $h_1=h_2 \circ \varphi_1 \circ \pi_2 \circ E_2$ or $h_1=h_2 \circ \widetilde{\varphi_1} \circ \pi_2 \circ E_2$) can be treated in the same way.
We have
$
f(0)=g\circ h_2 \circ \varphi_1 (0) = g \circ h_2(0)g \circ h_2(1)
\hbox{ and }
f(1)=g \circ h_2(0).
$
which contradicts the hypothesis.

Suppose $h_1 \in \{ \pi_2 \circ E_2  , \pi_2 \circ E_2 \circ E_2 \}$, then $h_1(0)=1$, $h_1(1)=0$ and
 $h_1(2)=\varepsilon$ or $h_1(0)=0$, $h_1(1)=1$ and
 $h_1(2)=\varepsilon$. In both case we easily check that $h$ is a unit of
 $\MSE_2$. This ends the first part of the proof.

\medskip

2) We now suppose $f$ is a prime morphism in $\MSE_{2}$ such that
 $|f(01)|_j > |f(01)|_k \geq |f(01)|_2$, where $A_3 = \{ j,k,2 \}$.

We proceed by contradiction. We suppose $f(1)$ is a prefix
of $f (0)$. The other case can be treated in the same way. There
exist $u$ and $v$ in $A_3^*$ such that $f(0) = uv$ and $f(1) = u$. We define
$g,h : A_3 \rightarrow A_3^*$ by $g(0) = u$, $g(1) = v$, $g(2) =
\varepsilon $, $h(0) = 01$, $h(1) = 02$ and $h (2) = \varepsilon$. We
remark that $h$ is not a unit and $f = g\circ h$.
To end the proof it suffices to show that $g$ is not a unit of
 $\MSE_2$. We start proving $g$ belongs to $\MSE_2$.

Let $x\in \WSE$. As in Example 1 we can prove that $h$ belongs to
 $\MSE_2$. \linebreak Consequently $h(x)$ belongs to $\WSE$. Moreover $f(x) = g(h(x))$ belongs to $\WSE$.\linebreak From the point (\ref{locally}) of Theorem \ref{maintheo}
 it comes that $g$ belongs to $\MSE_2$. From Lemma
 \ref{longueur} we have $|f (01)|=|g (010)| \geq 6$. Consequently $|f (01)|_0 + |f (01)|_1 + |f (01)|_2 \geq 6$. From the hypothesis it comes that $|g(010)|_j = |f(01)|_j \geq 3$. Hence \linebreak $2|g(01)|_j -|g(1)|_j \geq 3$ and $|g(01)|_j \geq 2$.

Now we prove by induction that for all $n\in \NN$ we have
$$
|\pi_2 \circ g^n (01) | \geq n+2,
\
|\pi_2 \circ g^n (0) | \geq 1,
\hbox{ and }
|\pi_2 \circ g^n (1) | \geq 1.
$$

This is true for $n=0$. We suppose it is true for $n\in \NN$. From Lemma
 \ref{longueur} we have
$$
|\pi_2 \circ g^{n+1} (01) |
=
|\pi_2 \circ g^{n} (g(01)) |
\geq
|\pi_2 \circ g^{n} (jjk) |
=
|\pi_2 \circ g^{n} (01)| + |\pi_2 \circ g^{n} (j) |
\geq n+3.
$$
Moreover, from Lemma \ref{longueur} in $g(j)$ occurs a letter $a\in \{ 0,1 \}$. Consequently,
$$
|\pi_2 \circ g^{n+1} (j) |
=
|\pi_2 \circ g^{n} (g(j)) |
\geq
|\pi_2 \circ g^{n} (a) |
\geq 1.
$$
We proceed in the same way for the letter $k$. This concludes the
 induction.
Therefore, it is clear $g$ is expansive. This concludes the proof.
\end{proof}

\section{The monoid $\MSE$ is not finitely generated}

\subsection{Some preliminary results}

To prove the point (\ref{notfinitely}) of Theorem \ref{maintheo} we need the following subset of $\MSE
$. Let $\MSEprime$ be the set of morphisms $f \in \MSE_2$ such that for
some $n\in \NN$
$$
\pi_2 \circ f \in F_n,
\pi_1 \circ f \in G_n
\hbox{ and }
\pi_0 \circ f \in H_n
\hbox{ where } F_n = \{\varphi_1, \widetilde{\varphi_1}\}^n ,
$$
$$
G_n =  E_0 \circ \{\varphi_1, \widetilde{\varphi_1}\} \circ
E_2 \circ \{\varphi_1, \widetilde{\varphi_1}\}^{n-1}
\hbox{ and }
H_n = E_2 \circ E_0 \circ \{\varphi_1,
\widetilde{\varphi_1}\}^{n-1}.
$$
With the two following lemmata we prove that $\MSEprime$ is not
empty. Before we need a new definition and we make some remarks.

Let $u \in \{0, 1\}^*$, $v \in \{0, 2\}^*$
 and $w \in \{1, 2\}^*$ be three words. We say that $u$, $v$ and $w$ {\it intercalate between them} if and only if there exists $x \in A_3^*$ such that $\pi_2(x)=u$, $\pi_1(x)=v$ and $\pi_0(x)=w$.

Let $( u_n)_{n\in \NN}$ be the Fibonacci word : $u_{n+1} = u_n +
u_{n-1}$ for all $n\geq 1$, $u_0 = 0$ and $u_1 = 1$.
We can remark that for all $n\geq 1$ we have
$$
M_{\varphi_1^n} = M_{\widetilde{\varphi}_1^n} = M_{\varphi_1}^n
=
\left[
\begin{array}{lll}
u_{n+1} & u_{n} & 0 \\
u_{n}     & u_{n-1} & 0\\
0 & 0 & 0
\end{array}
\right] .
$$

\begin{lemma}
Let $ n \geq 2$,
$f \in F_n$,
$g \in G_n$
and
$h \in H_n$.
Then, for all $a\in \{ 0,1 \}$ we have
$|f (a)|_0 = |g (a)|_0$,
$|f (a)|_1 = |h (a)|_1$
and
$|g (a)|_2 = |h(a)|_2$.
\end{lemma}
\begin{proof}
It suffices to remark that $M_f = M_{\varphi_1^n}$,
$$
M_g
=
\left[
\begin{array}{lll}
u_{n+1} & u_{n} & 0 \\
0       & 0     & 0 \\
u_{n-1}     & u_{n-2} & 0
\end{array}
\right]
\hbox{ and }
M_h
=
\left[
\begin{array}{lll}
0       & 0     & 0 \\
u_{n} & u_{n-1} & 0 \\
u_{n-1}     & u_{n-2} & 0
\end{array}
\right]
.
$$
\end{proof}

\begin{lemma}
\label{intercalate}
Let $f,g$ and $h$ be three morphisms from $A_3$ to $A_3^*$ such that
 $f(a), g(a)$ and $h(a)$ are respectively words on the alphabets $\{ 0 , 1\}$, $\{ 0,2\}$ and $\{ 1,2 \}$ for all $a\in A_3$. Then, $f(a)$, $g(a)$ and $h(a)$ intercalate between them for all $a\in A_3$ if and only if there exists a morphism $\psi : A_3 \to A_3^*$ such that
 $\pi_2 \circ \psi = f$, $\pi_1 \circ \psi = g$ and  $\pi_0 \circ \psi = h$.
\end{lemma}
\begin{proof}
For all $a\in A_3$ let $\psi( a)$ be the word obtained intercalating
 $f(a)$, $g(a)$ and $h(a)$. This defines a morphism $\psi : A_3 \to A_3^*$. We can
 check it satisfies $\pi_2 \circ \psi = f$, $\pi_1 \circ \psi = g$ and
 $\pi_0 \circ \psi = h$. The reciprocal is left to the reader.
\end{proof}

\begin{lemma}
\label{prefixsuffix}
For all $n\in \NN^*$, $\varphi_1^n (1)$ is a prefix but not a
 suffix of $\varphi_1^n (0)$. And for all $n\in \mathbb{N}^*\backslash\{1\}$ if $g = E_0 \circ \widetilde{\varphi_1} \circ E_2 \circ \widetilde{\varphi_1}^{n-1}$ then $g(1)$ is a suffix but not a prefix of $g(0)$.

\end{lemma}
\begin{proof}
Let $n\in \NN^*$. We have $\varphi_1^{n} (0) = \varphi_1^{n-1} (01) = \varphi_1^{n} (1)  \varphi_1^{n-1} (1)$. Hence $\varphi_1^n (1)$ is a prefix of $\varphi_1^n (0)$.

We proceed by induction to prove that $\varphi_1^n (1)$ is not
 suffix of $\varphi_1^n (0)$. For $n=1$ is it clear. Suppose it is true
 for $n\in \NN^*$. We prove it is also true for $n+1$.

We have  $\varphi_1^{n+1} (0) = \varphi_1^{n+1} (1)  \varphi_1^{n} (1)$
 and $\varphi_1^{n+1} (1) = \varphi_1^{n} (0)$. Suppose $\varphi_1^{n+1} (1)$ is a suffix of $\varphi_1^{n+1} (0)$. Looking at $M_{\varphi_1^{n}}$ we remark that $|\varphi_1^{n} (1)| < |\varphi_1^{n} (0)|$, therefore $\varphi_1^{n} (1)$ is a suffix of $\varphi_1^{n} (0)$ which contradicts the hypothesis. This concludes the first part of the proof. The other part can be achieved in the same way.
\end{proof}

\begin{lemma}
\label{existence}
Let $n\in \mathbb{N}^*$, $f_n = \varphi_1^n$, $g_n = E_0 \circ \widetilde{\varphi_1} \circ E_2 \circ \widetilde{\varphi_1}^{n-1}$ and $h_n= E_2 \circ E_0 \circ \widetilde{\varphi_1}^{n-1}$. Then there exists a morphism $\psi_n \in \MSE_2$ such that $\pi_2 \circ \psi_n = f_n $, $\pi_1 \circ \psi_n = g_n$ and $\pi_0 \circ \psi_n = h_n$.
\end{lemma}
\begin{proof}
We easily check that if $\psi $ is a morphism such that $\pi_2 \circ \psi = f_n$, $\pi_1 \circ \psi = g_n$ and $\pi_0 \circ \psi = h_{n}$, for some $n\in \NN$, then $\psi$ belongs to $\MSE_2$.

We proceed by induction on $n$ to prove what remains. For $n=1$, we have
\begin{center}
\begin{tabular}{lllllllllllllll}
$f_1 $ : & $ A_3^{*}$ & $ \longrightarrow  A_3^{*}$,       &
$g_1 $ : & $ A_3^{*}$ & $ \longrightarrow A_3^{*}$,        &
$h_1 $ : & $ A_3^{*}$ & $ \longrightarrow A_3^{*}$,        &
$\psi_1   $ : & $ A_3^{*}$ & $ \longrightarrow A_3^{*}$         \\
         & $ 0      $ & $ \longmapsto     01 $            &
         & $ 0      $ & $ \longmapsto     0  $            &
         & $ 0      $ & $ \longmapsto     1  $            &
         & $ 0      $ & $ \longmapsto     01 $            \\
         & $ 1      $ & $ \longmapsto     0  $            &
         & $ 1      $ & $ \longmapsto     20 $            &
         & $ 1      $ & $ \longmapsto     2  $            &
         & $ 1      $ & $ \longmapsto     20 $            \\
         & $ 2      $ & $ \longmapsto \varepsilon$        &
         & $ 2      $ & $ \longmapsto \varepsilon$        &
         & $ 2      $ & $ \longmapsto \varepsilon$        &
         & $ 2      $ & $ \longmapsto \varepsilon$.
\end {tabular}
\end{center}
The morphism $\psi_1 $ is such that $\pi_2 \circ \psi_1 = f_1$, $\pi_1\circ \psi_1 = g_1$ and $\pi_0 \circ \psi_1 = h_1$, and consequently $\psi_1 $ belongs to $\MSE_2$.
For $n = 2$, we have

\begin{center}
\begin{tabular}{lllllllllllllll}
$f_2 :$  & $ A_3^{*}$ & $ \longrightarrow  A_3^{*}$,       &
$g_2 :$  & $ A_3^{*}$ & $ \longrightarrow A_3^{*}$,        &
$h_2 :$  & $ A_3^{*}$ & $ \longrightarrow A_3^{*}$,        &
$\psi_2 : $ & $ A_3^{*}$ & $ \longrightarrow A_3^{*}$         \\
         & $ 0      $ & $ \longmapsto     010 $            &
         & $ 0      $ & $ \longmapsto     200 $            &
         & $ 0      $ & $ \longmapsto     21  $            &
         & $ 0      $ & $ \longmapsto     \hskip -1.1pt 2010$            \\
         & $ 1      $ & $ \longmapsto     01  $            &
         & $ 1      $ & $ \longmapsto     0   $            &
         & $ 1      $ & $ \longmapsto     1   $            &
         & $ 1      $ & $ \longmapsto     01  $            \\
         & $ 2      $ & $ \longmapsto \varepsilon$        &
         & $ 2      $ & $ \longmapsto \varepsilon$        &
         & $ 2      $ & $ \longmapsto \varepsilon$        &
         & $ 2      $ & $ \longmapsto \varepsilon$.
\end {tabular}
\end{center}

The morphism $\psi_2 $ is such that $\pi_2 \circ \psi_2 = f_2$, $\pi_1\circ \psi_2 = g_2$ and $\pi_0 \circ \psi_2 = h_2$, and consequently $\psi_2 $ belongs to $\MSE_2$.
Now we suppose the result is true for $n-1$ and $n\geq 2$. We prove it
 is also true for $n+1$.
We have

$$
\begin{array}{lllllllll}
f_{n+1}(0)  & = & f_{n-1}(0)f_{n-1}(1)f_{n-1}(0), & f_{n+1}(1)  & = & f_n(0),       & f_{n+1}(2)  & = & \varepsilon, \\
g_{n+1}(0)  & = & g_{n-1}(0)g_{n-1}(1)g_{n-1}(0), & g_{n+1}(1)  & = & g_n(0),       & g_{n+1}(2)  & = & \varepsilon, \\
h_{n+1}(0)  & = & h_{n-1}(0)h_{n-1}(1)h_{n-1}(0), & h_{n+1}(1)  & = & h_{n}(0) \ {\rm and}  & h_{n+1}(2)  & = & \varepsilon.
\end{array}
$$

From the induction hypothesis and Lemma \ref{intercalate} we know
 $f_{i}(a)$, $g_{i}(a)$, $h_{i}(a)$ intercalate between them for all
 $a\in A_3$ and all $i\in \{ n-1,n\}$. Consequently, using Lemma
 \ref{intercalate}, there is a morphism $\psi : A_3 \to A_3^*$ such that
$\pi_2 \circ \psi_{n+1} = f_{n+1} $, $\pi_1 \circ \psi_{n+1} = g_{n+1}$
 and $\pi_0 \circ \psi_{n+1} = h_{n+1}$.
\end{proof}

\begin{prop}
\label{psipremier}
For all $n \in \mathbb{N}^*$, the morphism $\psi_n $ defined in Lemma
 \ref{existence} is prime in $\MSE_2$.
\end{prop}
\begin{proof}
We keep the notations of Lemma \ref{existence}.
Let $n\in \NN^*$. From Proposition \ref{condpremier} it suffices to prove that $\psi_n (1)$ is neither a prefix nor a suffix of
 $\psi_n (0)$. We proceed by contradiction: We suppose $\psi_n (1)$ is a prefix or a
 suffix of $\psi_n (0)$.

Suppose that $\psi_n (1)$ is a prefix of $\psi_n (0)$. Then $\pi_1 \circ \psi_n (1)$ is a prefix of $\pi_1 \circ \psi_n (0)$ and consequently $g_n(1)$ is
 a prefix of $g_n(0)$. This contradicts Lemma \ref{prefixsuffix}.

Suppose that $\psi_n (1)$ is a suffix of $\psi_n (0)$. Then $\pi_2 \circ \psi_n (1)$ is a suffix of $\pi_2 \circ \psi_n (0)$ and consequently $f_n (1)$ is
 a suffix of $f_n(0)$. This contradicts Lemma \ref{prefixsuffix} and proves the lemma.
\end{proof}

\begin{corollaire}
The set $\MSE_2$ contains infinitely many primes.
\end{corollaire}
\begin{proof}
We left as an exercise to prove that for all $n\in \NN^*$ we have
 $\psi_n \not = \psi_{n+1}$, where $\psi_n$ is defined in Lemma
 \ref{existence}. Proposition \ref{psipremier} ends the proof.
\end{proof}

\subsection{Proof of the point (\ref{notfinitely}) of Theorem \ref{maintheo}}

We proceed by contradiction: We suppose there exists $F = \{ f_1, \dots ,
f_l\} \subset \MSE$ generating $\MSE$, i.e., all $g\in \MSE$ is a composition
of elements belonging to $F$.

\medskip

Let $N = \sup_{a\in A_3, 1\leq i\leq l} |f_i (a)| $, $(\psi_n )_{n\in
\NN}$ be the morphisms defined in Lemma \ref{existence} and $(u_n)_{n\in
\NN}$ be the
Fibonacci word defined in the previous section. We remark
$$
\lim_{n\rightarrow +\infty} \max_{a\in A_3} |\psi_n (a)|
\geq
\lim_{n\rightarrow +\infty} u_{n+1} =+\infty .
$$
We fix $n\in \NN$ such that $\max_{a\in A_3} |\psi_n (a)| > N$. By
hypothesis there exist $g_1, \dots , g_k$ in $F$ such that $\psi_n = g_1
\circ \cdots \circ g_k$. We set $h =g_2 \circ \cdots \circ g_k$. The morphism
$\psi_n$ belongs to $\MSE_a$ for some $a \in A_3$.  It implies
$\psi_n (2) = \psi_n (a) = \varepsilon$ and consequently $a=2$.

There exists $b\in A_3$ such that $g_1 \in \MSE_b$. We remark $\psi_n =
g_1\circ h = g_1 \circ \pi_b \circ h$. Two cases occurs.

\medskip

First case: For all $a\in \{ 0,1 \}$ we have $| \pi_b \circ h (a)| = 1$.

The morphism $h$ being a morphism with Sturmian erasures we cannot have \linebreak
$\pi_b \circ h (0) = \pi_b \circ h (1)$. Consequently $\psi_n = g_1\circ E_2$
or $\psi_n=g_1$. This implies there exists $a\in A_3$ such that $|g_1 (a)| >
N$ which is not possible.

\medskip

Second case: There exists $a\in \{ 0,1 \}$ such that $| \pi_b \circ h (a)| > 1$.

We remark $\pi_b \circ h = \pi_b \circ h\circ \pi_2$

If $b=2$ then $\pi_b \circ h_{|\{ 0,1 \}} : \{ 0,1 \} \to \{ 0,1 \}^* \subset A_3^*$ is a Sturmian morphism different
from $E$ and $Id_{\{ 0,1 \}}$. Hence from a remark we make in Subsection
\ref{almost} there exist $i$ and $j$ in $\{ 0,1 \}$, $i\not = j$, such that
the word $\pi_b \circ h_{|\{ 0,1 \}} (i)$ is a prefix or a suffix of $\pi_b
\circ h_{|\{ 0,1 \}} (j)$. Hence $\psi_n (i)$ is a prefix or a suffix of $\psi_n
(j)$. Proposition \ref{condpremier} implies $\psi_n$ is not prime in
$\MSE_2$ which contradicts Proposition \ref{psipremier}.

Let $b\not = 2$. We set $\{ b,c \} = \{ 0,1\}$.  Then, $E_c \circ \pi_b
\circ h_{|\{ 0,1 \}} :  \{ 0,1 \} \to \{ 0,1 \}^* \subset A_3^*$ is a Sturmian morphism different
from $E$ and $Id_{\{ 0,1 \}}$. Hence from a remark we make in Subsection
\ref{almost} there exist $i$ and $j$ in $\{ 0,1 \}$, $i\not = j$, such that
the word $E_c \circ \pi_b \circ h_{|\{ 0,1\}} (i)$ is a prefix or a suffix of
$E_c \circ \pi_b
\circ h_{|\{ 0,1 \}} (j)$. Hence $\psi_n (i)$ is a prefix or a suffix of $\psi_n
(j)$. Proposition \ref{condpremier} implies $\psi_n$ is not prime in
$\MSE_2$ which contradicts Proposition \ref{psipremier}.

This concludes the proof.

\section{Some further facts about words with Sturmian erasures}

\subsection{Geometrical remarks}

We recall that a Sturmian word can be viewed as a coding of a straight
half line in $\RR^2$ with direction $(1,\alpha)$ where $\alpha$
is a positive irrational number, or in other terms as a
trajectory of a ball in the game of billiards in the square with elastic reflexion on the boundary. We do not give the details here, we
refer the reader to \cite{Lo}.

Let us extend the construction given in \cite{Lo} to obtain what is
usually called {\it billiard words in the unit cube} $[0,1]^3$. Let
$d = (d_0,d_1 ,
d_2)\in [0, +\infty [^3$ and $\rho = ( \rho_0 , \rho_1 , \rho_2 ) \in [0,1 [^3 $.
Let $D$ be the half line with direction $d$
and intercept \linebreak $\rho$ that is to say $D = \{ td +  \rho ; t \geq 0
\}$. Consider the intersections of $D$ with the planes $x=a$, $y=a$,
$z=a$, $a\in \ZZ$: We denote by $I_0, I_1, \dots$ these consecutive intersection points .
   We say $I_n$ crosses the face  $F_i$, $i\in \{ 0,1,2 \}$, if the
   $i+1$-th coordinate of $I_n$ is an integer and the $i+1$-th
   coordinate of $I_{n+1} - I_n$ is not equal to $0$.

We set $\Omega_n = \{ i\in \{ 0,1,2 \} ; I_n \hbox{ crosses } F_i\}$.
Let $x= u_0 u_1 \dots$
be a word such that
$$
u_n =
\left\{
\begin{array}{ll}
i  & \hbox{ if } \Omega_n = \{ i \} , \\
ij & \hbox{ if } \Omega_n = \{ i,j \} \hbox{ where } i\not = j ,\\
ijk& \hbox{ if } \Omega_n = \{ i,j,k \} = \{ 0,1,2 \}.
\end{array}
\right.
$$

We say $x $
is a {\it billiard word in the unit cube} $[0,1]^3$ (with direction $d$ and
intercept $\rho $).
We can also say that $x$ is a coding of $D$. Of course a half line can
have several codings. One of the codings of a half line is periodic if
and only if $d\in \gamma \ZZ^3$ for some $\gamma \in \RR_+$.

When one of the coordinates of the
direction is equal to zero and the two others are rationally independent
we can easily deduce from \cite{Lo} (Chapter 2) that  $x$ is a Sturmian
word. The reciprocal is also true: All Sturmian words can be obtained in
this way (see \cite{Lo}).

We remark that if $x$ is a non-periodic cubic billiard word then $\pi_0 (x)$
is a Sturmian word with direction $d = (0 , d_1 ,
d_2)$ and intercept $\rho = (0, \rho_1, \rho_2)$ (i.e. the orthogonal
projection of $D$ onto $\{ 0 \} \times [0,+\infty [\times [0,+\infty
[$). We have the analogous remark for $\pi_1 (x)$ and $\pi_2 (x)$.
It is easy to conclude that a cubic billiard word is a word with
Sturmian erasures if and
only if it is non periodic and $d\in ] 0 , +\infty [^3$.

There exist words with Sturmian erasures that
are not cubic billiard words. For example, take the Fibonacci word $x$
(Example 1) and the morphism $\psi$ defined by $\psi (0) = 0012$ and
$\psi (1) = 01$. It is easy to see that $y=\psi (x)$ is a word with
Sturmian erasures. Let us show
that if $y$ was a cubic billiard word then the word 102 should appear in
$x$, which is not the case.

We briefly sketch the proof. Suppose $y$ is a cubic billiard word with
direction $d=(1,\alpha , \beta)$ and intercept $\rho$, then the words
$\pi_0 (y)$ and $\pi_2 (y)$ are Sturmian words with
respective directions $(0, \alpha , \beta) $ and $(1, \alpha , 0)$.
It can be shown that $\alpha = \theta -1$ and $\beta =
(\theta - 1)^2$ where $\theta = (\sqrt{5}+1)/2$. But with such a
direction $d=(1,\theta -1 , (\theta - 1)^2)$ easy calculus show that the
word 102 should appear in $y$.

\subsection{Balanced words}

Let us recall a characterization of Sturmian words due to Hedlund
and Morse \cite{HM2}. Let $A $ be a finite alphabet. We say a
word $x\in A^{\NN}$ is balanced if for all factors $u$ and $v$ of $x$
having the same length we have $||u|_a -
|v|_a| \leq 1$ for all $a\in A$.
Suppose $\card A = 2$. A word $x\in A^{\NN}$ is Sturmian if and only if
$x$ is non eventually periodic and balanced.
P. Hubert characterizes in \cite{Hu} the words on a three letters
alphabet that are balanced. This characterization shows that such words are not
words with Sturmian erasures.

\begin{defi}
Let $A$ be a finite alphabet.
We say $x \in A^{\NN}$ is $n$-balanced if $n$ is the least integer such
 that: For all words $u$ and $v$ appearing in $x$ and
 having the same length we have $\mid\mid u\mid_a-\mid v\mid_a\mid\leq n$ for all $a\in A$.
\end{defi}

Clearly, Sturmian words are $1$-balanced.

\begin{prop}
\label{ch5.46}
If $x\in A_3^{\NN }$ is a word with Sturmian erasures then $x$ is non eventually
 periodic and $2$-balanced.
\end{prop}
\begin{proof}
It is clear $x$ is non eventually periodic. From a previous remark we
 know that $x$ is not $1$-balanced.

Suppose $x$ is $n$-balanced with $n\geq 3$: There exist $e\in A_3$ and two words
 $u$ and $v$ appearing in $x$ and having the same length such that $| |u|_e-| v|_e|\geq
3$.

For all $a\in A_3$ we set $n(a) = || u|_a- | v|_a|$.
Then we can set $A_3 = \{ a,b,c\}$
 where $n(a)\geq 3$ and $n(a) \geq n(b) \geq n(c)$. Without loss of
 generality we suppose $| u|_a- | v|_a = n(a)$. As $|u| = |v|$ we have
$n(a) = (|v|_b -|u|_b) + (|v|_c -|u|_c)$. Consequently we necessarily
 have $|v|_b -|u|_b \geq 0$ and $|v|_c -|u|_c \geq 0$ because $n(a) \geq n(b) \geq n(c)$. Thus $ n(b) = |v|_b -|u|_b$, $ n(c) = |v|_c -|u|_c$ and \linebreak $n(a) = n(b) +n(c)$. We also see that $n(b)\geq 2$ and $n(c)\geq 0$.

Suppose there exists a factor $u'$ of the word $u$ verifying
 $|\pi_c(u')|=|\pi_c(v)|$ and \linebreak $|\pi_c(u')|_a-|\pi_c(v)|_a \geq 2$. Then
 this would say that $\pi_c (x)$ is not balanced and a fortiori not
 Sturmian which would end the proof.

Let us find such a $u'$.
We have $|\pi_c (u)| \geq |\pi_c (v)| \geq |v|_b\geq 2$. Hence there
 exists a non-empty word $u'$ satisfying $|\pi_c(u')|=|\pi_c(v)|$ and
 having an occurrence in $u$.
Moreover
$$
|\pi_c (u')|_a + |\pi_c (u')|_b
=
|\pi_c (u')|
=
|\pi_c (v)|
=
|v|-|v|_c
$$
$$
=
|u|-|v|_c
=
|u|_a + |u|_b + |u|_c - |v|_c
=
|\pi_c (u)|_a + |\pi_c (u)|_b -n(c) .
$$
Hence
$$
|\pi_c (u)|_a - n(c)
=
|\pi_c (u')|_a + |\pi_c (u')|_b -|\pi_c (u)|_b
\leq
|\pi_c (u')|_a
$$
and then
$$
2
\leq
n(b)
=
n(a)-n(c)
=
|\pi_c (u)|_a - |\pi_c (v)|_a  - n(c)
\leq
|\pi_c (u')|_a - |\pi_c (v)|_a,
$$
which ends the proof.
\end{proof}

\subsection{Complexity}

Let $x$ be a word with Sturmian erasures and $f$ be a morphism belonging to $\MSE_i$ for some $i\in A_3$. Then $\pi_i (x)$ is a Sturmian word and \linebreak $f (x) = f(\pi_i (x))$. Consequently from a result of Coven and Hedlund \cite{CH} we deduce there exist two integers $n_0$ and $k$ such that $P_x (n) = n+k$ for all $n\geq n_0$.

For example, let $F$ be the Fibonacci word and $f: \{ 0,1,2 \}
\rightarrow \{ 0,1,2 \}^*$ be the morphism defined by $f(0)=0102$,
$f(1)=01$ and $f(2)=\varepsilon$. It is a morphism with Sturmian
erasures and $y = f(F)$ is a word with Sturmian erasures. In fact it is
a cubic billiard word with direction $d = (1,\theta -1 , (\theta -
1)^2)$ and intercept $\rho = (0,\theta -1 , (\theta -
1)^2)$, where $\theta$ is the golden mean $(\sqrt{5}+1)/2$.

This does not contradict the result in \cite{AMST} saying that if
$1,\alpha $ and $ \beta $ are rationally independent then the
complexity of the cubic billiard word with direction $(1,\alpha , \beta
)$ and intercept $\rho \in ]0,1[^3$ is $n^2 + n + 1$, because
$-1+(\alpha - 1) + (\alpha - 1)^2 = 0$.

\subsection{Conclusion}
Many generalizations of the Sturmian words were tried
 (more letters, applications of $\mathbb{Z}^2$ to \{0, 1\}, ...) but none
 appeared to be entirely suitable in the sense that it seems impossible
 to extend these properties to a more general domain astonishing varieties of the properties characterizing these words. The example which we chose for
 this paper, does not derogate from this rule. Nevertheless, the fact
 that $\MSE$ is not given by a finite generator shows a
 fundamental difference between the Sturmian words and any
 generalization with more than two letters because the definition
 adopted here was less ``compromising'' possible. Furthermore, this
 definition gives a words of a complexity structurally similar to the one of
 the Sturmian words.

%%%%%%%%%%%%%%%%%%%%%%%%%%%%%%%%%%%%%%%%%%%%%%%%%%%%%%%%%%%%%%%%%%%%
% Bibliographie                                                    %
%%%%%%%%%%%%%%%%%%%%%%%%%%%%%%%%%%%%%%%%%%%%%%%%%%%%%%%%%%%%%%%%%%%%

\end{document}